\documentclass[11pt]{article}

\usepackage[margin=1.0in]{geometry}
\usepackage{amsmath,amssymb,amsthm,mathtools}
\usepackage{enumitem}
\usepackage{microtype}
\usepackage[hidelinks]{hyperref}

\newtheorem{theorem}{Theorem}[section]
\newtheorem{question}[theorem]{Question}
\newtheorem{corollary}[theorem]{Corollary}
\newtheorem{lemma}[theorem]{Lemma}
\theoremstyle{remark}
\newtheorem{remark}[theorem]{Remark}

\title{Non-colliding billiards in the plane}
\author{Itai Benjamini \and Alexander Shamov \and Barak Weiss}
\date{May 27, 2026}

\begin{document}
\maketitle

\begin{abstract}
We present an open problem about non-colliding freely moving hard disks in the Euclidean plane, together with related positive and negative partial results. The open problem is stated in a non-degenerate form: velocities are required to be pairwise distinct and their speeds are required to be uniformly bounded away from infinity. The positive deterministic result gives a bounded, injective, non-colliding velocity assignment for the integer lattice; after a common velocity shift, the speeds are also bounded away from zero. The negative result shows that no bounded continuous vector field on the whole plane can serve as a universal assignment satisfying the same separation inequality for all pairs of points at distance greater than one. We also record a space-time interpretation of the problem, relate it to packings by nonparallel cylinders in three dimensions, and formulate a corresponding topological-dynamical question for cylinder packings.
\end{abstract}

\section{Introduction}

A set $S \subset \mathbb R^2$ is called \emph{uniformly discrete} if
\[
\inf_{\substack{x,y\in S\\ x\ne y}} \|x-y\| >0 .
\]
By scaling the plane, one may normalize this lower bound to be at least $1$.

Consider a uniformly discrete set of initial positions in the Euclidean plane. To each point $x$ assign a fixed velocity vector $v(x)$, and let all particles evolve simultaneously by
\[
x(t)=x+t v(x), \qquad t\in \mathbb R .
\]
We ask when the moving configuration can remain uniformly discrete for all positive and negative times. Equivalently, if one replaces the points by disks of radius $\rho>0$, we ask when the disks can move forever without intersecting.

The non-degenerate version considered here imposes two additional requirements: distinct particles should have distinct velocities, and all speeds should lie in a deterministic compact subinterval of $(0,\infty)$. The upper bound is a bounded-speed assumption. The lower bound is not forced by the collision condition, and it is not invariant under arbitrary Galilean shifts. Its purpose is to exclude static particles and to ask for a genuinely moving configuration. In deterministic examples one can often enforce such a lower bound by adding the same velocity vector to every particle, since only relative velocities affect collisions. In an isotropic random problem, however, adding a fixed nonzero vector would generally destroy rotation invariance, so the lower bound has to be part of the formulation. One may also ask weaker variants in which this lower bound is omitted.

The basic ergodic version already seems open. The isotropic version asks for the same properties with an additional rotation-invariance assumption.

\medskip
\noindent\textbf{Open problem (ergodic version).}
Does there exist a probability measure on marked point configurations $(S,v)$, where $S\subset\mathbb R^2$ is uniformly discrete and $v:S\to\mathbb R^2$ assigns a velocity to each point, with the following properties?
\begin{itemize}[leftmargin=2em]
\item \emph{Non-degenerate velocities:} there is  a deterministic constant $0<M<\infty$ such that, almost surely,
\[
\|v(x)\|\le M \qquad \text{for every }x\in S,
\]
and
\[
v(x)\ne v(y) \qquad \text{for every distinct }x,y\in S.
\]
\item \emph{Time invariance:} the law is invariant under the time evolution
\[
(S,v)\mapsto (S_t,v_t), \qquad
S_t:=\{x+t v(x):x\in S\},
\]
where the velocity attached to the moved point $x+t v(x)$ is still $v(x)$.
\item \emph{Translation invariance:} the law is invariant under translations of the plane, acting on the positions and leaving the velocity vectors unchanged.
\item \emph{Ergodicity:} the law is ergodic with respect to translations.
\item \emph{Non-emptiness:} $S$ is almost surely non-empty.
\item \emph{Hard-core condition at all times:} almost surely, for every $t\in\mathbb R$,
\[
\inf_{\substack{x,y\in S\\ x\ne y}}
\|(x+t v(x))-(y+t v(y))\|\ge 1 .
\]
\end{itemize}

\noindent\textbf{Isotropic version.}
Does such a measure exist if, in addition, its law is invariant under rotations of the plane, where rotations act simultaneously on positions and velocity vectors? 

The first item is part of the non-degenerate formulation, not an optional strengthening. Without some non-degeneracy condition, a stationary hard-core point process with all velocities equal to zero gives a trivial solution. Pairwise distinct velocities rule out this particular degeneracy, while the lower speed bound additionally rules out particles whose individual motion degenerates to rest.

It is useful to keep the following space-time picture in mind. Each particle determines a worldline
\[
L_x:=\{(x+t v(x),t):t\in\mathbb R\}\subset\mathbb R^2\times\mathbb R .
\]
The hard-core condition says that every horizontal time slice meets these worldlines in a point configuration with mutual distances at least $1$. Time invariance asks that the law of the sliced marked configuration be unchanged when all worldlines are translated in the time direction. This viewpoint also connects the problem to packings by infinite cylinders in three dimensions.

Let $e_1,e_2,e_3$ denote the standard basis of $\mathbb R^3$, let $B_\rho\subset \operatorname{span}(e_1,e_2)$ be the closed Euclidean disk of radius $\rho$, and let
\[
Z_\rho:=B_\rho+\mathbb R e_3
\]
be the closed vertical cylinder of radius $\rho$. A circular cylinder of radius $\rho$ is an isometric image of $Z_\rho$, and its direction is the image of the vertical direction.

\begin{lemma}[Worldlines give cylinder packings]\label{lem:worldline-cylinder}
Suppose that a moving configuration satisfies the hard-core condition above and that $\|v(x)\|\le M$ for every $x\in S$. Then any two distinct worldlines $L_x,L_y\subset\mathbb R^3$ have Euclidean distance at least
\[
\frac{1}{\sqrt{1+M^2}}.
\]
Consequently, the Euclidean cylinders of any radius
\[
0<\rho\le \frac{1}{2\sqrt{1+M^2}}
\]
about the worldlines have pairwise disjoint interiors. If the velocities are pairwise distinct, then these cylinders are pairwise nonparallel. If moreover $m\le \|v(x)\|\le M$, then their directions lie in the compact annulus of directions making angle between $\arctan m$ and $\arctan M$ with the $e_3$-axis.
\end{lemma}

\begin{proof}
If $M=0$, then all worldlines are vertical and the assertion follows directly from the hard-core condition. Hence assume $M>0$.

Fix $x\ne y$ and points $(x+t v(x),t)\in L_x$ and $(y+s v(y),s)\in L_y$. Put $\tau=s-t$ and
\[
h=(x+t v(x))-(y+t v(y)).
\]
By the hard-core condition, $\|h\|\ge 1$. The difference of the two space-time points is
\[
\bigl(h-\tau v(y),-\tau\bigr).
\]
If $|\tau|\ge 1/M$, then its Euclidean norm is at least $1/M$, which is at least $1/\sqrt{1+M^2}$. If $|\tau|\le 1/M$, then
\[
\bigl\|\bigl(h-\tau v(y),-\tau\bigr)\bigr\|^2
\ge (1-M|\tau|)^2+|\tau|^2
\ge \frac{1}{1+M^2},
\]
since the minimum of $(1-Mu)^2+u^2$ for $u\ge 0$ occurs at $u=M/(1+M^2)$ and equals $1/(1+M^2)$. This proves the distance bound.

The cylinder conclusion follows immediately. Finally, the direction of $L_x$ is the line spanned by $(v(x),1)$. If two such directions agree, then $(v(x),1)=\lambda(v(y),1)$ for some nonzero scalar $\lambda$; comparing the third coordinate gives $\lambda=1$, and hence $v(x)=v(y)$. Thus pairwise distinct velocities give nonparallel cylinders. The angle $\theta$ between $(v(x),1)$ and the $e_3$-axis satisfies $\tan\theta=\|v(x)\|$, giving the stated annulus of directions.
\end{proof}

The cylinder-packing viewpoint places the problem near a classical line of work. Bezdek and W. Kuperberg determined the maximal density of a packing by congruent infinite circular cylinders in $\mathbb R^3$; the value is $\pi/\sqrt{12}$, as in the planar disk-packing problem, and it is achieved by parallel cylinders whose intersections with a perpendicular plane form an optimal disk packing \cite{BezdekKuperberg}. K. Kuperberg constructed nonparallel cylinder packings of positive density \cite{KuperbergNonparallel}. Further developments include dense nonparallel packings of Ismailescu and Laskawiec \cite{IsmailescuLaskawiec} and recent work of Eliyahu on lower and upper densities for nonparallel cylinder packings \cite{Eliyahu}.

This suggests the following topological-dynamical analogue of the
ergodic problem above. Let $\mathcal{CL}(\mathbb R^3)$ be the space of
closed subsets of $\mathbb R^3$ with the Chabauty-Fell topology. Since
$\mathbb R^3$ is locally compact and second countable, this space is
compact and metrizable. For background on the Fell topology and
related topologies, see Fell's original paper \cite{Fell} and Beer's
monograph \cite{Beer}, and see \cite[Appendix]{SolomonSmilansky} for an explicit
and convenient metric inducing the topology. Let $\mathcal P_\rho$ be the collection of closed subsets of $\mathbb R^3$ that are unions of closed circular cylinders of radius $\rho$ with pairwise disjoint interiors, and let
\[
\mathcal{CYL}_\rho:=\overline{\mathcal P_\rho}^{\,\mathcal{CL}(\mathbb R^3)}.
\]
Then $\mathcal{CYL}_\rho$ is a compact invariant space for the natural action of the Euclidean group, and also for the subgroup of translations.

\medskip
\noindent\textbf{Open problem (topological cylinder version).}
For the actions of the translation group and of the full Euclidean isometry group on $\mathcal{CYL}_\rho$:
\begin{itemize}[leftmargin=2em]
\item classify, or describe, the minimal closed invariant subsets;
\item decide whether there is a nontrivial minimal subsystem all of whose nonempty elements are packings by cylinders with no two cylinder axes parallel.
\end{itemize}
The empty configuration is a fixed point, hence gives a trivial minimal subsystem. The question is whether minimality can coexist with a genuinely nonparallel cylinder-packing structure. Known nonparallel packing constructions show that positive density and nonparallelity are compatible, but orbit closures may lose the features one wants to preserve, for instance by acquiring degenerate limits such as the empty configuration or packings containing two cylinders with parallel axes.

For background on uniformly discrete sets in the plane and related dynamics, see, for example, Solomon and Weiss \cite{SolomonWeiss}. A well-known problem about distances between moving points is the lonely runner problem; see Tao \cite{TaoLonelyRunner}.

\section{A construction with the integer lattice as initial positions}

Throughout the paper, $\mathbb R^2$ is equipped with its standard inner product $\langle\cdot,\cdot\rangle$ and Euclidean norm $\|\cdot\|$. Let
\[
I=\begin{pmatrix}0&-1\\ 1&0\end{pmatrix}
\]
denote rotation by $\pi/2$.

Suppose that two particles start at $x,y\in\mathbb R^2$ and move with different velocities $v(x),v(y)\in\mathbb R^2$. Then
\begin{align*}
\inf_{t\in\mathbb R}\|(x+t v(x))-(y+t v(y))\|
&=\inf_{t\in\mathbb R}\|(x-y)+t(v(x)-v(y))\| \\
&=\frac{|\langle x-y, I(v(x)-v(y))\rangle|}{\|v(x)-v(y)\|} .
\end{align*}
This is the distance from the point $x-y$ to the line spanned by $v(x)-v(y)$.

Set
\[
w(x):=I v(x).
\]
Since $I$ is an isometry, the preceding formula shows that a uniform lower bound $r>0$ on all two-particle closest-approach distances is equivalent to
\[
\bigl|\langle x-y,w(x)-w(y)\rangle\bigr|\ge r\|w(x)-w(y)\|
\]
for all distinct initial positions $x,y$ under consideration. Equivalently, the closest-approach distance for the pair is the quotient
\[
\frac{\bigl|\langle x-y,w(x)-w(y)\rangle\bigr|}{\|w(x)-w(y)\|}.
\]
Thus, to prove a positive uniform lower bound, it is enough to prove the slightly stronger strict inequality
\begin{equation}\label{eq:separation-condition}
\bigl|\langle x-y,w(x)-w(y)\rangle\bigr|>c\|w(x)-w(y)\|
\end{equation}
for some constant $c>0$. Conversely, \eqref{eq:separation-condition} implies that the closest-approach distance for every pair is greater than $c$. The injectivity of $w$ is equivalent to pairwise distinct velocities, and boundedness of $w$ is equivalent to boundedness of the velocities.

\begin{question}\label{q:discrete}
Let $S\subset\mathbb R^2$ be uniformly discrete. Does there exist a bounded injective map $w:S\to\mathbb R^2$ and a constant $c>0$ such that
\[
\bigl|\langle x-y,w(x)-w(y)\rangle\bigr|>c\|w(x)-w(y)\|
\]
for every pair of distinct points $x,y\in S$?
\end{question}

For $S=\mathbb Z^2$ the answer is positive.

\begin{theorem}\label{thm:lattice}
Let $\phi:\mathbb Z\to\mathbb R$ be bounded and strictly increasing. Define
\[
w:\mathbb Z^2\to\mathbb R^2, \qquad
w(x_1,x_2)=(\phi(x_1),\phi(x_2)).
\]
Then $w$ is bounded and injective, and it satisfies \eqref{eq:separation-condition} with every constant $c\in(0,1)$.
\end{theorem}

\begin{proof}
Boundedness of $w$ follows from boundedness of $\phi$. Since $\phi$ is strictly increasing, it is injective; hence $w$ is injective.

Let $x=(x_1,x_2)$ and $y=(y_1,y_2)$ be distinct points of $\mathbb Z^2$. For each coordinate $i\in\{1,2\}$, either $x_i=y_i$, or else $|x_i-y_i|\ge 1$. Since $\phi$ is increasing, $x_i-y_i$ and $\phi(x_i)-\phi(y_i)$ have the same sign. Therefore
\[
(x_i-y_i)(\phi(x_i)-\phi(y_i))
=|x_i-y_i|\,|\phi(x_i)-\phi(y_i)|
\ge |\phi(x_i)-\phi(y_i)|.
\]
Summing over $i=1,2$ gives
\begin{align*}
\langle x-y,w(x)-w(y)\rangle
&=\sum_{i=1}^2 (x_i-y_i)(\phi(x_i)-\phi(y_i)) \\
&\ge \sum_{i=1}^2 |\phi(x_i)-\phi(y_i)| \\
&\ge \left(\sum_{i=1}^2 |\phi(x_i)-\phi(y_i)|^2\right)^{1/2} \\
&=\|w(x)-w(y)\|.
\end{align*}
Because $w$ is injective, $\|w(x)-w(y)\|>0$ whenever $x\ne y$. Hence, for every $c\in(0,1)$,
\[
\bigl|\langle x-y,w(x)-w(y)\rangle\bigr|
\ge \|w(x)-w(y)\|
> c\|w(x)-w(y)\|,
\]
which is \eqref{eq:separation-condition}.
\end{proof}

\begin{remark}
The monotone-coordinate idea in Theorem \ref{thm:lattice} is closely related in spirit to constructions used in the study of nonparallel cylinder packings, especially the construction of K. Kuperberg \cite{KuperbergNonparallel}. The statement here is included because it gives a direct and elementary collision-avoidance construction for the lattice $\mathbb Z^2$.
\end{remark}

\begin{corollary}\label{cor:lattice-billiards}
For the initial configuration $\mathbb Z^2$, there are bounded, pairwise different velocities for which all closest-approach distances are bounded below by a positive constant. Consequently, disks of sufficiently small radius centered initially at the points of $\mathbb Z^2$ can move for all $t\in\mathbb R$ without colliding. The velocities may also be chosen so that their speeds are uniformly bounded away from zero.
\end{corollary}

\begin{proof}
Choose $\phi$ as in Theorem \ref{thm:lattice}, for instance $\phi(n)=\arctan n$, and set $v=-I w$. Theorem \ref{thm:lattice} gives a positive lower bound on all closest-approach distances. Taking the disk radius smaller than half this lower bound prevents collisions. The velocities are bounded and pairwise different because $w$ is bounded and injective.

Finally, replacing every velocity $v(x)$ by $v(x)+a$, where $a\in\mathbb R^2$ is fixed, does not change any relative velocity and hence does not change any closest-approach distance. Choosing $\|a\|>\sup_x\|v(x)\|$ makes all speeds bounded below by $\|a\|-\sup_x\|v(x)\|>0$, while they remain bounded above by $\|a\|+\sup_x\|v(x)\|$.
\end{proof}

\section{No universal continuous vector field}

The construction above uses the product structure of $\mathbb Z^2$ and does not appear to extend directly to arbitrary uniformly discrete sets. The following theorem shows that this limitation is unavoidable in one natural sense: there is no bounded continuous field on the whole plane that satisfies the separation inequality simultaneously for all pairs of points at distance greater than one. The theorem is not an obstruction to the random problem stated above, which concerns a discrete random set of initial positions rather than a continuous assignment on all of $\mathbb R^2$; its role is to rule out one possible universal construction.

\begin{theorem}\label{thm:no-universal}
There is no bounded continuous function $w:\mathbb R^2\to\mathbb R^2$ for which there exists a constant $c>0$ such that
\begin{equation}\label{eq:universal-abs}
\bigl|\langle x-y,w(x)-w(y)\rangle\bigr|>c\|w(x)-w(y)\|
\end{equation}
for every pair $x,y\in\mathbb R^2$ with $\|x-y\|>1$.
\end{theorem}

\begin{proof}
Assume, for a contradiction, that such a function $w$ exists.

Let
\[
D:=\{(x,y)\in\mathbb R^2\times\mathbb R^2:\|x-y\|>1\}.
\]
This set is connected. Indeed, under the change of variables $(x,y)\mapsto ((x+y)/2,x-y)$, it is the product of $\mathbb R^2$ with the connected set $\{z\in\mathbb R^2:\|z\|>1\}$.

The inequality \eqref{eq:universal-abs} implies that
\[
\langle x-y,w(x)-w(y)\rangle\ne 0
\]
for every $(x,y)\in D$. Since this inner product is continuous on the connected set $D$, its sign is constant. Replacing $w$ by $-w$ if necessary, we may assume that
\begin{equation}\label{eq:universal-positive}
\langle x-y,w(x)-w(y)\rangle>c\|w(x)-w(y)\|
\end{equation}
for all $(x,y)\in D$.

Set
\[
a:=\min\{c/2,1/2\}\in(0,1),
\qquad
\delta:=\arcsin a>0.
\]
For every nonzero vector $u\in\mathbb R^2$, define the closed cone
\[
C_u:=\{z\in\mathbb R^2:\langle u,z\rangle\ge a\|u\|\,\|z\|\}.
\]
Equivalently, $C_u$ consists of all vectors whose angle with $u$ is at most
\[
\arccos a=\frac{\pi}{2}-\delta .
\]
In particular, each $C_u$ is a closed convex cone, and $C_{\lambda u}=C_u$ for every $\lambda>0$.

We first record a consequence of \eqref{eq:universal-positive}. If $1<\|x-y\|\le 2$, then
\[
c\ge a\|x-y\|,
\]
and therefore \eqref{eq:universal-positive} implies
\begin{equation}\label{eq:cone-near}
w(x)-w(y)\in C_{x-y}.
\end{equation}
We claim that the same conclusion holds for all $x,y$ with $\|x-y\|>1$:
\begin{equation}\label{eq:cone-all}
w(x)-w(y)\in C_{x-y}.
\end{equation}
Let $L=\|x-y\|>1$. Choose an integer $n$ such that
\[
L/2\le n<L;
\]
such an integer exists, for example $n=\lceil L/2\rceil$. Then $L/n\in(1,2]$. Put
\[
z_i=x+\frac{i}{n}(y-x), \qquad i=0,1,\ldots,n.
\]
For each $i$, the segment length $\|z_i-z_{i+1}\|=L/n$ lies in $(1,2]$, and $z_i-z_{i+1}$ is a positive scalar multiple of $x-y$. Hence \eqref{eq:cone-near} gives
\[
w(z_i)-w(z_{i+1})\in C_{x-y}.
\]
Since $C_{x-y}$ is a convex cone,
\[
w(x)-w(y)=\sum_{i=0}^{n-1}\bigl(w(z_i)-w(z_{i+1})\bigr)
\in C_{x-y},
\]
proving \eqref{eq:cone-all}.

Now consider the cluster set of $w$ at infinity,
\[
V_\infty:=\bigcap_{R>0}
\overline{\,w\bigl(\{x\in\mathbb R^2:\|x\|\ge R\}\bigr)\,}.
\]
Here and below, the bar denotes closure. Because $w$ is bounded, each set in the intersection is compact. It is also connected: the exterior $\{x:\|x\|\ge R\}$ is connected, its image under the continuous map $w$ is connected, and the closure of a connected set is connected. These compact connected sets are nested as $R$ increases. Therefore $V_\infty$ is nonempty, compact, and connected.

We shall use the following elementary characterization: a point $\xi\in\mathbb R^2$ lies in $V_\infty$ if and only if there is a sequence $x_n\in\mathbb R^2$ such that $\|x_n\|\to\infty$ and $w(x_n)\to\xi$. Indeed, one implication is immediate from the definition. Conversely, if $\xi\in V_\infty$, then for each $n$ we may choose $x_n$ with $\|x_n\|\ge n$ and $\|w(x_n)-\xi\|<1/n$.

The set $V_\infty$ is not a singleton. Suppose instead that $V_\infty=\{\xi\}$. Then $w(x)\to\xi$ as $\|x\|\to\infty$; otherwise, boundedness of $w$ would produce a sequence tending to infinity along which $w$ has a subsequential limit different from $\xi$, contradicting the definition of $V_\infty$.

Fix $y\in\mathbb R^2$ and a unit vector $u\in S^1$. Applying \eqref{eq:cone-all} to $x=ru$ and this fixed $y$, and then letting $r\to\infty$, gives
\[
\xi-w(y)\in C_u.
\]
Since this holds for every $u\in S^1$, it follows that $\xi-w(y)=0$; for example, the conditions for $u$ and $-u$ can both hold only for the zero vector. Thus $w(y)=\xi$ for every $y$, so $w$ is constant. This contradicts the strict inequality \eqref{eq:universal-abs}. Hence $V_\infty$ is a nontrivial connected compact set. In particular, it contains arbitrarily many distinct points.

It remains to obtain a contradiction from the angular restrictions imposed by \eqref{eq:cone-all}. Let $\mathbb T=\mathbb R/2\pi\mathbb Z$, and let $d_{\mathbb T}$ denote circular distance on $\mathbb T$. For a nonzero vector $z$, write $\arg z\in\mathbb T$ for its direction.

We claim that if $\xi,\eta\in V_\infty$ are distinct, and if there are sequences $x_n,y_n$ with
\[
\|x_n\|\to\infty,
\qquad
\|y_n\|\to\infty,
\qquad
w(x_n)\to\xi,
\qquad
w(y_n)\to\eta,
\]
such that
\[
\arg x_n\to\alpha,
\qquad
\arg y_n\to\beta
\]
for some $\alpha,\beta\in\mathbb T$, then
\begin{equation}\label{eq:angular-separation}
d_{\mathbb T}(\alpha,\beta)\ge 2\delta .
\end{equation}

To prove the claim, choose indices $N(n)$ tending to infinity so fast that
\[
\frac{\min\{\|x_{N(n)}\|,\|y_{N(n)}\|\}}
{\max\{\|x_n\|,\|y_n\|\}}
\longrightarrow \infty .
\]
Then
\[
\arg (y_{N(n)}-x_n)\to\beta,
\qquad
\arg (x_{N(n)}-y_n)\to\alpha .
\]
For all large $n$, the relevant distances are greater than $1$, so \eqref{eq:cone-all} gives
\[
w(y_{N(n)})-w(x_n)\in C_{y_{N(n)}-x_n},
\qquad
w(x_{N(n)})-w(y_n)\in C_{x_{N(n)}-y_n}.
\]
Let $q:=\eta-\xi\ne 0$. Since
\[
w(y_{N(n)})-w(x_n)\to q,
\qquad
w(x_{N(n)})-w(y_n)\to -q,
\]
the preceding cone inclusions imply
\[
d_{\mathbb T}(\arg (y_{N(n)}-x_n),\arg q)
\le \frac{\pi}{2}-\delta+o(1),
\]
and
\[
d_{\mathbb T}(\arg (x_{N(n)}-y_n),\arg(-q))
\le \frac{\pi}{2}-\delta+o(1).
\]
Because $d_{\mathbb T}(\arg q,\arg(-q))=\pi$, the triangle inequality on the circle yields
\[
d_{\mathbb T}(\arg (y_{N(n)}-x_n),\arg (x_{N(n)}-y_n))
\ge 2\delta-o(1).
\]
Letting $n\to\infty$ proves \eqref{eq:angular-separation}.

Now choose an integer $k>\pi/\delta$ and choose distinct points
\[
\xi_1,\ldots,\xi_k\in V_\infty .
\]
For each $i$, by the sequence characterization of $V_\infty$ and compactness of the circle, choose a sequence $x_n^{(i)}$ such that
\[
\|x_n^{(i)}\|\to\infty,
\qquad
w(x_n^{(i)})\to\xi_i,
\qquad
\arg x_n^{(i)}\to\alpha_i\in\mathbb T .
\]
The claim gives
\[
d_{\mathbb T}(\alpha_i,\alpha_j)\ge 2\delta
\]
for all $i\ne j$. But the circle of circumference $2\pi$ cannot contain more than $\pi/\delta$ points whose pairwise circular distances are all at least $2\delta$: the open arcs of radius $\delta$ around such points are pairwise disjoint. This contradicts $k>\pi/\delta$.

The contradiction proves the theorem.
\end{proof}

\begin{remark}
One can ask analogous questions in other spaces. For example, in the hyperbolic plane $\mathbb H^2$, particles could move along geodesics with prescribed tangent vectors. The construction in Theorem \ref{thm:lattice} uses the affine product structure of $\mathbb Z^2\subset\mathbb R^2$ and does not transfer directly to $\mathbb H^2$. It would be interesting to know whether there is an isometry-invariant analogue of the open problem, or an analogue of Theorem \ref{thm:lattice}, in the hyperbolic plane.
\end{remark}

\end{document}